\documentclass[12pt,reqno]{amsart}

\usepackage{amssymb}
\usepackage{amsmath}
\usepackage{latexsym}
\usepackage{amsthm}
\usepackage{epsfig}
\usepackage{color}
\usepackage{comment}
\usepackage{amsfonts,amssymb,mathrsfs,amscd}

\usepackage{enumitem}
\usepackage{bm}

\newtheorem{proposition}{Proposition}[section]
\newtheorem{theorem}[proposition]{Theorem}
\newtheorem{corollary}[proposition]{Corollary}
\newtheorem{lemma}[proposition]{Lemma}
\theoremstyle{definition}
\newtheorem{definition}[proposition]{Definition}
\newtheorem{remark}[proposition]{Remark}

\numberwithin{equation}{section}
\newcommand{\Tr}{\mathop\mathrm{Tr}}
\def\R{\mathbb R}
\def\eb{\varepsilon}
\def\({\left(}
\def\){\right)}
\def\Cal{\mathcal}

\def\Dt{\partial_t}
\def\Tr{\operatorname{Tr}}

\begin{document}
 \title[Attractors for  damped wave equations]
 {Sharp bounds on the attractor dimensions  for  damped wave equations}
\author[ A.A. Ilyin, A.G. Kostianko,
 S.V. Zelik] {
Alexei Ilyin${}^{2,4,5}$, Anna Kostianko${}^{1,5}$, and
 Sergey Zelik${}^{1,2,3,5}$}
\email{ilyin@keldysh.ru}
\email{a.kostianko@imperial.ac.uk}
\email{s.zelik@surrey.ac.uk}
\address{${}^1$ Zhejiang Normal University, Department of Mathematics, Zhejiang, China}
\address{${}^2$ Keldysh Institute of Applied Mathematics, Moscow, Russia}
\address{${}^3$ University of Surrey, Department of Mathematics, Guildford, UK}
\address{${}^4$ Institute for Information Transmission Problems, Moscow, Russia}
\address{${}^5$ HSE University, Nizhny Novgorod, Russia}
\begin{abstract}
We give the explicit estimates of order $\gamma^{-d}$
(with logarithmic correction in the 1D case)  for the
fractal  dimension of the attractor
of the damped hyperbolic equation (or system) in a bounded domain
$\Omega\subset \mathbb R^d$, $d\ge1$ with linear damping coefficient $\gamma>0$.
 The key ingredient in the proof
for $d\ge3$
is Lieb's bound for the $L_p$-norms of systems with orthonormal gradients
based on the Cwikel--Lieb--Rozenblum (CLR) inequality for  negative
eigenvalues of the Schr\"odinder operator. The case $d=1$
is simpler, but contains a logarithmic correction term that seems to be inevitable.
The 2D case is  more difficult and is strongly based on the Strichartz-type
estimates for the linear equation.
Lower bounds of the same order  for the dimension
of the attractor   are also obtained for a damped hyperbolic system
with nonlinearity containing  a small  non-gradient perturbation term, meaning that
in this case our estimates are optimal for $d\ge2$ and contain a logarithmic
discrepancy for $d=1$. Estimates for the various dimensions
(Hausdorff, fractal, Lyapunov) of the attractor in purely gradient case are also given.
We show, in particular, that the Lyapunov dimension of
a non-trivial attractor is of the order $\gamma^{-1}$ in all
spatial dimensions $d\ge1$.
\end{abstract}

\maketitle
\tableofcontents
 \setcounter{equation}{0}
\section{Introduction}\label{sec1}

The theory of global attractors of dissipative evolution PDEs
 has been actively developing over the last 50 years.
The corresponding literature is vast. From the large collection
of monographs we mention  the  classical ones  \cite{BV,T};
the current state of the art of many aspects of the theory is discussed
in a recent review~\cite{ZUMN}.
The development of the theory was significantly motivated
by the need to analyze  the asymptotic (as $t\to\infty$)
behavior of solutions of the  Navier--Stokes system,
which still  remains in the focus of the theory.

Another important popular example is the hyperbolic equation with
dissipation, which was first studied from the  point of view of
global attractors in~\cite{BV-MS}. The main difference between this
equation and the Navier--Stokes system and other nonlinear
parabolic equations is that the solution operators are not
smoothing (compact) for $t>0$. That is why the attractor of the
hyperbolic equation, constructed in \cite{BV-MS}, initially had the
property of attracting bounded sets in the phase space only in the
weak topology.

The further important progress  was made in~\cite{Aro}, where it
was proved that the semigroup of solution operators is
asymptotically compact. After proving this fact, the global
attractor (in the strong topology) is constructed in the usual way
as the $\omega$-limit set of the absorbing ball.

Estimates of the dimension of this attractor  were first obtained
in~\cite{Ghid-Tem} (see also \cite {T}) by using the technique of
global Lyapunov exponents. Some refinements to these estimates were
obtained in \cite{Huang, Shen}. We point out  that these estimates
are given in a rather implicit form, and it is difficult to trace
the dependence of the estimates on the physical parameters of the
system. In addition, in all previous works known to
the authors, only the case of a gradient nonlinearity has been
considered. This gives the extra structure, namely, the existence
of a global Lyapunov function which, in turn, drastically changes
the behavior of the attractor dimensions as the dissipation
coefficient tends to zero (see Section \ref{S:sec5}) and this has
been completely overseen in the previous studies.

In the case of the Navier--Stokes the key role in estimating the
global Lyapunov exponents in the phase $L_2$ is played by
the (dual) Lieb--Thirring inequalities for $L_2$-orthonormal
systems \cite{Lieb, LT, T85, BV,T, Lieb90}. Recently, the authors studied
a  regularized model in incompressible hydrodynamics
(the so-called simplified Euler--Bardina model) where the natural phase space
is $H^1$ and, accordingly,  the global Lyapunov exponents
are estimated in $H^1$. The $L_p$-estimates for $H^1$-orthonormal systems
in the subcritical case
\cite{L} have proved to be very useful and made it possible to
obtain optimal two-sided estimates (as the regularization parameter
$\alpha\to0$) for the dimension of
the global attractors for various boundary conditions
both in dimension two and three \cite{IZLap70, IKZPhysD}.

In this work we study the dimension estimates for the global attractors of
 the damped  hyperbolic system  with
nonlinearity of Sobolev growth posed in a bounded spatial domain
 $\Omega\subset\mathbb R^d$, $d\ge1$.  Again, the  $L_p$-estimates for families of
functions with suborthonormal gradients (this time in the critical case)
fit very nicely in the theory for $d\ge3$  producing explicit  estimates
for  the dimension.  In particular, we present here one more nontrivial application for
the celebrated Cwikel--Lieb--Rozenblum (CLR) inequality.

The lower dimensional case $d=1$ is simpler, while the case $d=2$  is
strongly based on the Strichartz type estimates, see Section \ref{S:sec4}.

 Furthermore, for a non-gradient
damped hyperbolic system
lower bounds of the same order  for the fractal dimension
of the attractor   are also obtained.

 In the gradient case, we utilize the global
Lyapunov function and give sharp upper and lower bounds for
the Lyapunov dimension of the attractor giving the explicit
formula for it in terms of the spectral properties  of the equilibria,
see Section \ref{S:sec5}. We also give explicit formulas for
the Hausdorff dimension in a generic case, however, there still is
an essential gap between the natural upper and lower bounds for the
fractal dimension in the gradient case.
\par

To describe the results in greater detail and put them in perspective
let us consider in a bounded domain $\Omega\subset\mathbb R^d$ the
damped hyperbolic equation/system
\begin{equation}\label{hypintro}
\aligned
\partial_t^2u+\gamma \partial_tu-&\Delta u+f(u)=g,\\
&u\vert_{\partial \Omega}=0,\\
u(0)=u_0,\ &\partial_tu(0)=v_0.
\endaligned
\end{equation}
Here $u=(u_1,\dots,u_N)$, $N\ge1$,  is the unknown vector function,
the damping coefficient $\gamma>0$ is a small parameter, $g=g(x)\in
 L_2(\Omega)$. The non-linear vector function $f$ is of the
gradient form with some small  non-gradient perturbations specified
below and satisfies the standard dissipativity assumptions and its
derivative has Sobolev growth rate:
\begin{equation}\label{f'growth}
|f'(u)|\le C(1+|u|^{2/(d-2)})
\end{equation}
 if $d\ge3$ (for $d=1$ no growth restrictions are required
  and for $d=2$ we pose the polynomial growth restriction where the
   growth exponent may be arbitrary).

It is known (see, for instance, \cite{BV,T}) that system \eqref{hypintro}
is well posed in the phase space
$E= H^1_0(\Omega)\times L_2(\Omega)$ so that the
semigroup of the solution operators
$$
S(t)\{u_0,v_0\}=\{u(t),\partial_t u(t)\}
$$
is well defined. The semigroup $S(t)$ is dissipative in $E$, that is, it possesses
 an absorbing ball of radius $R_0$ in $E$, where $R_0$ depends on $\|g\|_{L_2}$
 and the
constants describing the structure and the rate of growth of the
non-linear function $f$. The key result for us is that $R_0$ is  bounded uniformly with respect to
 $\gamma$ as $\gamma\to0^+$.

 The semigroup $S(t)$  has a global attractor
$\mathscr A\Subset E$, that is a compact  strictly invariant set
uniformly attracting bounded sets in $E$: for every $\delta>0$
$$
S(t)B\subset \mathscr{O}_\delta(\mathscr{A})\ \text{for}\ t\ge T_0(B,\delta),
$$
where $\mathscr{O}_\delta(\mathscr{A})$ is an arbitrary
$\delta$-neighbourhood of $\mathscr{A}$ in $E$ and $B\subset E$ is
a bounded set in $E$.
Furthermore, the attractor $\mathscr A$ has finite fractal dimension~\cite{Ghid-Tem,T},
 and our main result in Section~\ref{S:sec3} is
  Theorem~\ref{T:main-d} in which the following upper bound
for the dimension is proved for $d\ge3$ (see~\eqref{dimd}):
\begin{equation}\label{dimintro}
\dim_F\mathscr A\le N\frac{\mathrm c_d}{\gamma^d}\cdot B_d^d,
\quad B_d:=\sup_{\{u,\partial_tu\}\in\mathscr A} \|f'(u)\|_{L_d(\Omega)},
\end{equation}
where ${\mathrm c_d}$ is dimensionless constant depending on $d$ only, whose explicit
expression is given in terms of the CLR constant in Section~\ref{S:sec7} below,
 and where $B_d$ is bounded as $\gamma\to 0^+$ in view of~\eqref{f'growth}
and the Sobolev inequality, since the attractor is bounded in $E$
uniformly as $\gamma\to0^+$. As already mentioned, similar upper
bound for the dimension holds also for $d=1,2$, however, with
rather different proof.

Thus, we have the estimate for the attractor dimension of the form
 $\dim_F\mathscr A\preceq \gamma^{-d}$ as $\gamma\to0^+$ (to be precise, there is a logarithmic correction
 for $d=1$), and it makes sense to look at the optimality
 of this upper bound.

An important observation in this connection is that the upper bound of order
$\gamma^{-d}$ can be supplemented with the lower bound of the
same order only in the case of a system, more precisely,
in the case of a non-gradient system.

 In light of the above, estimate \eqref{dimintro}
 (along with its analogues for $d=1,2$)
is interesting for a system
with non-gradient nonlinearity only. Our results can be summarized as
follows
$$
\aligned
d=1\quad \gamma^{-1}\preceq&\dim_F\mathscr A\preceq\gamma^{-1}\ln(\gamma^{-1}),\\
d\ge2\quad \gamma^{-d}\preceq &\dim_F\mathscr A\preceq \gamma^{-d},
\endaligned
$$
where, as usual the upper bounds are universal (for the considered  classes of
non-linear functions $f(u)$), while the lower bounds hold for a specially chosen
families of functions within these classes).

 For the gradient case,
we develop more appropriate technique which gives much better
upper bounds, namely,
\begin{equation}\label{grad}
\dim_H\mathscr A\le\dim_F\mathscr A\le \dim_L\mathscr A\le \frac C\gamma,
\end{equation}
which is independent of the space dimension $d$ (the constant $C$
is also independent of $\gamma\to0^+$). Here and below $H$, $F$ and
$L$ stand for the Hausdorff, fractal (box-counting) and Lyapunov
dimensions of the attractor, respectively. Moreover, in the generic
case where all of the equilibria are hyperbolic (the case of
regular attractors in terminology of Babin and Vishik), the
Hausdorff dimension of the attractor $\mathscr A$ remains bounded
as $\gamma\to0$, so \eqref{grad} gives the upper bounds of
different order for the fractal dimension of the attractor
$\mathscr A$. At the moment, we do not know which of these bounds
is sharp although the examples of regular attractors where the
fractal dimension is strictly larger than the Hausdorff one have
been recently constructed in \cite{ZPochin}.

 The paper is organized as follows.
In  Section~\ref{S:sec2} we specify the conditions on
the non-linear function $f(u)$,  define the solution semigroup
$S(t)$ and prove that it is dissipative paying special
attention to the fact that the radius of the
absorbing ball (containing the attractor) is
bounded uniformly as $\gamma\to0^+$. We also recall the standard machinery
of the volume contraction method for estimating
the dimension of invariant sets.

In  Section~\ref{S:sec3} we derive the upper bound~\eqref{dimintro}
for $d\ge3$.

In Section~\ref{S:sec4} the upper bounds of the type  $\dim_F\mathscr A\preceq\gamma^{-1}\ln(1/\gamma)$
and $\dim_F\mathscr A\preceq\gamma^{-2}$ are proved for $d=1$ and $d=2$, respectively,
involving a different approach for
estimating the numbers $q(n)$.

In Section~\ref{S:sec5} it is shown that if $g=0$ and the matrix
$f'(0)$ has a complex eigenvalue, then the dimension of
the global attractor of a damped hyperbolic system with at least $N\ge2$
components
 admits a lower bound of the order $O(1/\gamma^d)$.
This settles the problem on optimal bounds for the fractal dimension
of the global attractor
of a damped hyperbolic system with non-gradient nonlinearity.

The case of a gradient nonlinearity is treated in Section~\ref{S:sec6}.

Finally, in Section~\ref{S:sec7} we prove $L_p$-bounds for
systems of functions (and vector functions) with suborthonormal
gradients that play the key role for the upper bounds for the dimension.

 \setcounter{equation}{0}
\section{Preliminaries: damped hyperbolic equations, attractors and dimensions}
\label{S:sec2}

In a bounded domain $\Omega\subset\mathbb R^d$ we consider the
damped hyperbolic system
\begin{equation}\label{hyp}
\aligned
&\partial_t^2u+\gamma \partial_tu-\Delta u+f(u)=g,\\
&\qquad u\vert_{\partial \Omega}=0,\\
&\xi_u\big|_{t=0}=\xi_0,\quad\xi_u:=\{u,\partial_t u\}.
\endaligned
\end{equation}
Here $u=(u_1,\dots,u_N)$, $N\ge1$,  is the unknown vector function,
the damping coefficient $\gamma>0$ is a small parameter, $g=g(x)\in  L_2(\Omega)$.
The non-linear vector function $f$ is of the form
\begin{equation}\label{f(u)}
f(u)=\Phi(u)+f_\gamma(u),
\end{equation}
where the leading term $\Phi(u)$ is of the gradient form and is independent of $\gamma$,
and the non-gradient perturbation $f_\gamma(u)$ is small and may depend on $\gamma$:
\begin{equation}\label{condfb}
\Phi(u)=\nabla_uF_0(u),\quad |f_\gamma(u)|\le\gamma,\quad  |f'_\gamma(u)|\le K.
\end{equation}
We further suppose that the leading term $\Phi$ satisfies the
following structure and growth conditions
\begin{eqnarray}
 \Phi(u)\cdot u&\ge&  F_0(u)-C, \label {11}\\
  F_0(u)&\ge& -\frac12(\lambda_1-\delta)|u|^2-C,  \label {22} \\
  |\Phi'(u)|&\le& C(1+|u|^{2/(d-2)}),  \label{33}
\end{eqnarray}
where $\lambda_1$ is the first eigenvalue of the Dirichlet
 Laplacian $-\Delta^D_{\Omega}$ and
$\delta>0$. The last growth condition is assumed for $d\ge3$. For
$d=2$, we accept any polynomial growth of $\Phi$ and for $d=1$ no
growth restrictions are needed.
\par
In \eqref{condfb}, \eqref{33} and in what follows
the notation $h'(u)$ for a vector  function $h(u)$ is used
for the matrix $\{\partial_{u_j}h_i(u)\}_{i,j=1}^N$ with norm
$|h'(u)|$
$$
|h'(u)|^2:=\sum_{i,j=1}^N(\partial_{u_j}h_i(u))^2.
$$
As usual, equation \eqref{hyp} is considered in the energy phase space
$$
\xi_u\in E:= H^1_0(\Omega)\times  L_2(\Omega),
$$
with norm
$$
\|\xi_u\|^2_E:=\|\nabla u\|^2_{L_2}+\|\partial_tu\|^2_{L_2}.
$$

\par
It is well-known that under the above assumptions equation
\eqref{hyp} is well-posed in the energy phase space, generates a
dissipative solution semigroup $S(t):E\to E$ in it via
$S(t)\xi_u(0):=\xi_u(t)$ and possesses a global attractor $\mathscr
A=\mathscr A_\gamma$ in $E$, see \cite{BV,T} and references
therein. It is crucial in what follows that the $E$-norm of the
attractor $\mathscr A_\gamma$ remains bounded as $\gamma\to0$ (in
contrast to the higher energy norms which may grow in a badly
controllable way), so we will avoid the usage of higher energy
norms in what follows trying to estimate the corresponding
dimensions via the quantities related with the energy norm only.
For the convenience of the reader, we give below the derivation of
the basic dissipative estimate indicating the dependence on
$\gamma$.
\begin{proposition}\label{Prop1.sol}
Let conditions \eqref{f(u)}--\eqref{33} hold.
Then the solution semigroup $S(t):\xi_u(0)\to\xi_u(t)$
corresponding to \eqref{hyp} is dissipative for all sufficiently
small $\gamma\le\gamma_0=\delta/(8\lambda_1))$ and  possesses an
absorbing ball in $E$: for any $R>0$
$$
S(t)B_E(R)\subset B_E(R_0)\quad \text{for}\quad t\ge T(R,R_0)
$$
where $R_0=C(\|g\|^2_{L_2}+1)$, and the constant $C$ is independent of
$\gamma$ and depends on the structure and growth constants in
\eqref{condfb}--\eqref{33}.
\end{proposition}

\begin{proof} Indeed, multiplying equation \eqref{hyp} by $\Dt u+\eb u$,
 where $1>\varepsilon>0$ is a small positive number of
 order $\gamma$  (see  \cite{BV,T} for the justification of this
 multiplication) and integrating over $x\in\Omega$, we arrive at
$$
 \frac d{dt}\Psi(\xi_u(t))+\Psi_1(\xi_u(t))=\eb(g,u)-(f_\gamma(u),\Dt u+\eb u)
 =:R(\xi_u(t)),
$$
where using that $\varepsilon$ is small and that
$|f_\gamma(u)|\le\gamma$, we have
$$
\aligned
 &\Psi(\xi_u):=\frac12\|\xi_u\|^2_E+\varepsilon(u,\partial_tu)+
(F_0(u),1)+\frac12\gamma\eb\|u\|^2_{L_2}-(g,u)\\&\qquad\qquad\le
\frac13\|\xi_u\|^2_E+(F_0,1)+C\|g\|^2_{L_2},\\
&\Psi_1(\xi_u):=(\gamma-\eb)\|\Dt u\|^2_{L_2}+\eb\|\nabla u\|^2_{L_2}
+\varepsilon(\Phi(u),u)\ge\\&\qquad\qquad\varepsilon\|\xi_u\|^2_E+\varepsilon(\Phi(u),u),\\
&R(\xi_u)\le\left(\frac\varepsilon 4+\frac\gamma 8\right)\|\xi_u\|^2_E
+ C\gamma(\|g\|^2_{L_2}+1).
\endaligned
$$

Next, setting  $\varepsilon:=\gamma/2$ we obtain
$$
\frac d{dt}\Psi(\xi_u)
+\frac\gamma4\left(\|\xi_u\|^2_E+2(\Phi(u),u)\right)\le
C\gamma(\|g\|^2_{L_2}+1).
$$
Using \eqref{11} we see that
\begin{multline*}
\frac\gamma4\left(\|\xi_u\|^2_E+2(\Phi(u),u)\right)\ge
\frac\gamma4\left(\|\xi_u\|^2_E+2(F_0(u),1)-C)\right)\ge\\\ge
\frac\gamma4\left(\|\xi_u\|^2_E+2\Psi(\xi_u)-\frac23\|\xi_u\|^2_E-C(\|g\|^2_{L_2}+1)\right)\ge\\\ge
\frac\gamma2\Psi(\xi_u)-C\gamma(\|g\|^2_{L_2}\!+1),
\end{multline*}
and, finally,
$$
\frac d{dt}\Psi(\xi_u(t))+\frac\gamma 2\Psi(\xi_u(t))\le \gamma C_1(\|g\|^2_{L_2}+1).
$$
By Gronwall's inequality this gives that the set
$$
B:=\{\xi_u\in E, \Psi(\xi_u)\le4 C_1(\|g\|^2_{L_2}+1)\}
$$
is an absorbing set for the solution semigroup $S(t):E\to E$.

To prove the existence of an absorbing ball in $E$ we write for
$\Psi(\xi_u)$ with $\varepsilon=\gamma/2$ and an arbitrary $\mu>0$
$$
\aligned
\Psi(\xi_u)\ge \left(1/2-\mu-\gamma\right)\|\xi\|^2_E
+(F_0(u),1)-c(\mu)(\|g\|^2_{L_2}+1)\ge\\\ge
\left(1/2-\mu-\gamma\right)\|\xi\|^2_E
-\frac{\lambda_1-\delta}{2\lambda_1}\|\nabla u\|^2_{L_2}-c(\mu)(\|g\|^2_{L_2}+1)=\\=
(\delta/(2\lambda_1)-\mu-\gamma)\|\xi_u\|_E^2-c(\mu)(\|g\|^2_{L_2}+1).
\endaligned
$$
Fixing $\mu=\delta/(4\lambda_1)$ we see that for all
$\gamma\le\gamma_0=\delta/(8\lambda_1))$ the ball in $E$ centered
at the origin with radius
$$
R_0=\frac{4C_1+c(\delta/(4\lambda_1)}{\delta/(8\lambda_1)}(\|g\|^2_{L_2}+1)
$$
is an absorbing ball in $E$.
\end{proof}
\begin{remark}
Although the radius of the absorbing ball is uniform with respect
to $\gamma\le\gamma_0$, the time of entering is of the order
$1/\gamma$ as $\gamma\to0^+$.  We also mention
that the condition that $\gamma$ is small is not necessary
and is posed for simplicity only (since we only interested in
the behavior of the attractor as $\gamma\to0$). Slightly more
accurate estimates show that the radius of the absorbing ball
remains bounded as $\gamma\to\infty$ as well (if $f_\gamma$ remains
bounded as $\gamma\to\infty$), so it is actually uniform with respect
to $\gamma\in\R_+$. We also mention that conditions \eqref{11} and
\eqref{22} can be replaced by the standard one
$$
\Phi(u)\cdot u\ge-(\lambda_1-\delta)|u|^2-C,
$$
for some $\delta>0$.
\end{remark}

We continue with recalling the definitions which
are necessary for estimating the dimension of the
attractor via the volume contraction method~\cite{T}.

\begin{definition}\label{def-Lyap} Let $\mathscr A$
be a compact invariant with respect to the semigroup $S(t):E\to E$,
where $E$ is a Hilbert space. Assume also that $S(t)$
is uniformly (quasi)-differentiable on $\mathscr A$ with
the (quasi)differential $S'(\xi,t)\in\Cal L(E,E)$ and let
the map $\xi\to S'(\xi,t)$ be continuous for every fixed
$t\ge0$ as a map from $\mathscr A$ to $\Cal L(E,E)$.
For every $n\in\mathbb N$, we denote by
$$
\omega_n(S'(\xi,t)):=\|\Lambda^nS'(\xi,t)\|_{\Cal L(\Lambda^n E,\Lambda^n E)}
$$
the norm of $n$th exterior power of the operator $S'(\xi,t)$.
Then, due to the cocycle property,
$$
\omega_n(S'(\xi,t+h))\le\omega_n(S'(\xi,t))\omega_n(S'(S(t)\xi,h)),
$$
see e.g. \cite{T}. Define the volume contraction factor on $\mathscr A$ via
$$
\omega_n(\mathscr A,t):=\sup_{\xi\in\mathscr A}\omega_n(S'(\xi,t)).
$$
Then, due to the cocycle property and sub-multiplicativity, we have
$$
\omega_n(S,\mathscr A):=\lim_{t\to\infty}[\omega_n(\mathscr A,t)]^{1/t}=
\inf_{t\ge0}[\omega_n(\mathscr A,t)]^{1/t}.
$$
Finally, we extend the function $n\to\omega_n(S,\mathscr A)$ to
non-integer values of $n=n_0+s$, $n_0\in\mathbb Z_+$, $0<s<1$
setting
$$
\omega_n(S,\mathscr A):=
\omega_{n_0}(S,\mathscr A)^{1-s}\omega_{n_0+1}(S,\mathscr A)^s
$$
and define the Lyapunov dimension of the set $\mathscr A$ with respect to
semigroup~$S(t)$:
$$
\dim_{L}(S,\mathscr A):=\sup\{n\in\R_+\,:\ \omega_n(S,\mathscr A)\ge1\},
$$
 see, for instance,  \cite{KR} for
the modern exposition of the theory of Lyapunov exponents and Lyapunov dimension
in the finite dimensional case.
\end{definition}

Then, the key theorem of the volume
 contraction method concerning the fractal dimension
  has an especially elegant form~\cite{Ch-I}:
$$
\dim_F \mathscr A\le \dim_L(S,\mathscr A).
$$
 To the best of our knowledge, the corresponding  formula
for the Hausdorff dimension
$$
\dim_H \mathscr A\le \dim_L(S,\mathscr A).
$$
 has been first obtained  in
 \cite{DO} (see \cite{CF85,CFT} for the infinite-dimensional case).
 The case of fractal dimension is due to \cite{Hunt}
 (see \cite{blin,Ch-I} for the infinite-dimensional case).
 Note that the estimate on the fractal dimension has been
 obtained earlier in \cite{Ch-I2001,Ch-V-book} under some
 extra restrictions on the Lyapunov exponents
 (note also that the assumption of the continuity of the
 quasi-differential can be relaxed under these extra assumptions as
 well as in the case of Hausdorff dimension, but we do not know whether it
 is necessary in a general case).

At the next step, we recall the analogue of
 the Liouville formula which allows us to estimate
 the volume contraction factors through the appropriate
 traces of the linearized operators related with the
 equation of variations of~\eqref{hyp}. Namely, assume
 in addition that the differential $\xi_v(t):=S'(u,t)\xi_0$
 is given by the solution of the equation of variations
$$
\frac d{dt}\xi_v=\Cal L(\xi_u(t))\xi_v,\ \ \xi_v\big|_{t=0}=\xi_0,
$$
where $\Cal L(\xi_u(t))$ is a
linear operator in $E$. Then~\cite{T},
$$
\omega_n(S,\mathscr A)\le e^{q(n)},
$$
where the numbers $q(n)$
\begin{equation}\label{q(n)}
q(n):=\liminf_{t\to\infty}\frac1t\sup_{\xi_u\in \mathscr A}
\int_0^t\operatorname{Tr}_n\Cal L(\xi_u(s))\,ds
\end{equation}
majorate the  sums of the first $n$ global Lyapunov exponents. The
key term here is the  $n$-trace
$\operatorname{Tr}_n \Cal L$ defined as follows:
$$
\operatorname{Tr}_n\Cal L:=\sup_{\{\xi_i\}_{i=1}^n\in E}
\sum_{i=1}^n(\Cal L\xi_i,\xi_i)_E,
$$
where the supremum is taken with respect to all orthonormal systems
$\{\xi_i\}_{i=1}^n\in D(\Cal L)$, $(\xi_i,\xi_j)_E=\delta_{ij}$.

Thus, if we prove that $q(n)<0$ for some $n\in\mathbb N$,
then $\dim_F\mathscr A<n$. Therefore, for obtaining the
upper bounds for the attractor dimension, it is sufficient
to estimate the quantities $q(n)$.

Returning to hyperbolic system~\eqref{hyp},
we follow  the standard procedure and  write
 the equation of variations for equation \eqref{hyp} in
 terms of the
 variables $\xi_v:=\{v,\Dt v+\eb v\}\in E$ and use the
 standard metric in $E$ (this is the same as to use
 $\xi_v:=\{v,\Dt v\}$, but to introduce  the
 corresponding equivalent metric in the space $E$;
 although the volume contraction factors are
 independent of the choice of the metric, it
 is easier to estimate them in such a metric).
 In these new variables, the equation of variations reads
\begin{equation}\label{hypsys}
\partial_t\xi_v+\Lambda_\varepsilon \xi_v+F'(u(t))\xi_v=0,
\end{equation}
where
$$
\Lambda_\varepsilon=
\left[
\begin{array}{cc}
                              \varepsilon I & -I \\
                             -\Delta-\varepsilon(\gamma-\varepsilon)I & (\gamma-\varepsilon)I \\
                            \end{array}
\right],\ F'(u(t))=\left[
                \begin{array}{c}
                  0 \\
                  f'(u(t)) \\
                \end{array}
              \right].
$$
The phase space for~\eqref{hypsys} is the space $E$
with scalar product
$$
(y_1,y_2)_E=(\nabla\varphi_1,\nabla\varphi_2)+(\psi_1,\psi_2),\ \ y_i=\{\varphi_i,\psi_i\}.
$$
Then
$$
\aligned
(\Lambda_\varepsilon y,y)_E&=\varepsilon\|\nabla\varphi\|^2_{L_2}+(\gamma-\varepsilon)\|\psi\|^2_{L_2}
-\varepsilon(\gamma-\varepsilon)(\varphi,\psi)\ge\\&\ge\frac\varepsilon 2(\|\nabla\varphi\|^2_{L_2}+\|\psi\|^2_{L_2}),
\endaligned
$$
provided that $\varepsilon$ satisfies~\cite{T}
$$
0<\varepsilon\le\varepsilon_0,
\quad\varepsilon_0=\min\left(\frac\gamma 4,\frac{\lambda_1}{2\gamma}\right)\,,
$$
where $\lambda_1$ is the first eigenvalue of the Dirichlet Laplacian $-\Delta^D_{\Omega}$ in $\Omega$.
\par
In what follows, we will always get the upper bounds for the attractor
 dimension by estimating the traces of the operator
$$
\Cal L_\eb(u(t)):=-\Lambda_\eb-F'(u(t))
$$
in the phase space $E$ with the standard metric. The validity
of all formulas mentioned above for our damped wave equation \eqref{hyp}
is verified, for example, in \cite{T} or \cite{BV}.
\par
To conclude this preliminary section, we also recall that
 the lower bounds for the attractor dimension are usually
obtained utilizing the fact that the unstable manifold of any
equilibrium belongs to the attractor. Namely, if the semigroup
is $C^1$-smooth and possesses an equilibrium $\xi_0\in \Cal R$ such that
$$
\operatorname{ind}_+(\xi_0):=\#\{\mu_i\in \sigma(S'(\xi_0,1)),\ |\mu_i|>1\}=M,
$$
 then $\dim_F\mathscr A\ge M$, see, for instance, \cite{BV}.
  In this paper, we will always use this method for the lower
  bounds of the attractor dimension although there is an
  interesting alternative method, which is based on the
  homoclinic bifurcation theory, and which is developed
  exactly for damped wave equations, see \cite{TZ}.

\section{The upper bounds for the dimension: the case $d\ge3$}\label{S:sec3}
In this section, we start our consideration of the upper bounds
for the attractors dimension for damped wave equations.
Although, at the end of the day, the estimates look similar
for the lower dimensional $d\le2$ and the higher dimensional
$d\ge3$ cases, their proofs are surprisingly different and the
higher dimensional case where we start with is, in a sense, simpler
(at least than the 2D case).
The lower dimensional case will be considered later.

The dependence of our estimate on the nonlinearity $f$
will be accumulated in the following quantity:
$$
B_d:=\sup_{\{u,\partial_tu\}\in\mathscr A}  \|f'(u)\|_{L_d(\Omega)},
$$
which is bounded uniformly with respect to
$\gamma\to0^+$, in view of Proposition~\ref{Prop1.sol},
the growth condition~\eqref{33}
and the Sobolev inequality.

The main result of this section is the following theorem.

\begin{theorem}\label{T:main-d} Let $d\ge3$ and let
 conditions \eqref{f(u)}--\eqref{33} hold.
Then the fractal dimension of the global attractor $\mathscr A\Subset E$
of the semigroup corresponding to~\eqref{hyp} possesses the upper bound:
\begin{equation}\label{dimd}
\dim_F\mathscr A\le  N\frac{\mathrm c_d}{\gamma^d}\cdot B_d^d,\quad
\text{where}\quad
\mathrm c_d=8^d\left(\frac d{d-2}\right)^{d/2}
L_{0,d},
\end{equation}
and where $L_{0,d}$ is the Cwikel--Lieb--Rozenblum constant (see Section~\ref{S:sec7}).
\end{theorem}

\begin{proof}
We consider the equivalent linearized system~\eqref{hypsys}   on a solution
$u(t)$ lying on the attractor $\mathscr A$  and estimate the $n$-trace
of the operator $\Cal L_\eb(u(t))$:
\begin{equation}\label{anyd}
\aligned
\sum_{j=1}^n(\mathcal L_\eb(u(t)) \Theta_j, \Theta_j)_E=
-\sum_{j=1}^n(\Lambda_\varepsilon\Theta_j, \Theta_j)_E-\sum_{j=1}^n
(F'(u(t))\Theta_j, \Theta_j)_E=\\=
-\sum_{j=1}^n(\Lambda_\varepsilon\Theta_j, \Theta_j)_E
-\sum_{j=1}^n\int_\Omega  \langle f'(u(t,x))\varphi_j(x),\psi_j(x)\rangle\,dx\le\\
\le
-\frac\varepsilon 2n+
\sum_{j=1}^n\int_\Omega  |f'(u(t,x))|\,|\varphi_j(x)|\,|\psi_j(x)|\,dx\le\\\le
-\frac\varepsilon 2n+
\int_\Omega  |f'(u(t,x))|\rho_\varphi(x)^{1/2}\rho_\psi(x)^{1/2}\,dx,
\endaligned
\end{equation}
 where $\langle\cdot,\cdot\rangle$  denotes the scalar product in
 $\mathbb R^N$, and where we set
$$
\rho_\varphi(x):=\sum_{j=1}^n|\varphi_j(x)|^2,
\quad \rho_\psi(x):=\sum_{j=1}^n|\psi_j(x)|^2.
$$
Since $\Theta_j$'s, $\Theta_i=(\varphi_i,\psi_i)$, are orthonormal,
both systems $\{\nabla\varphi_j\}_{j=1}^n$ and $\{\psi_j\}_{j=1}^n$
are suborthonormal in $L_2$ and therefore for all $d\ge1$
\begin{equation}\label{rhopsi}
\|\rho_\psi\|_{L_1}\le n,
\end{equation}
while the estimates for $\rho_\varphi$ will depend on the dimension $d$.

We use  in~\eqref{anyd} H\"older's inequality
with exponents $d$, $2d/(d-2)$, $2$
 to obtain
\begin{equation}
\label{2}
\sum_{j=1}^n(\mathcal L(t) \Theta_j, \Theta_j)_E
\le
-\frac\varepsilon 2n+B_d\|\rho_\varphi\|_{L_{d/(d-2)}}^{1/2}\|\rho_\psi\|_{L_1}^{1/2}.
\end{equation}
We have \eqref{rhopsi} as before,
while Theorem~\ref{T:Lieb_suborth} gives that
$$
\|\rho\|_{L_{d/(d-2)}}\le (NL_{0,d})^{2/d}\frac d{d-2}n^{(d-2)/d}.
$$
The right-hand side in~\eqref{2} is independent of $t$ and of the
particular trajectory on the attractor. Therefore $q(n)$ satisfies
the same upper bound.
We obtain, as a result, that
$$
q(n)\le
-\frac\varepsilon 2n+B_d(d/({d-2}))^{1/2}(NL_{0,d})^{1/d}\cdot n^{1-1/d}.
$$
Since  $\gamma$ is small, we can
set $\varepsilon=\gamma/4$ here.
Then the number $n^*= N\gamma^{-d}{\mathrm c_d} B_d^d$ is
such that $q(n)<0$ for $n>n^*$ and therefore $n^*$ is an upper bound
both for the
Hausdorff \cite{BV,T} and the fractal \cite{Ch-I2001,Ch-I,Ch-V-book} dimension
of the global
attractor~$\mathscr A$. The proof is complete.
\end{proof}
\begin{remark} Analyzing the proof of the main theorem,
we see that the estimate for the dimension can be improved
by replacing the quantity $B_d$ by a weaker one:
$$
\widetilde B_d:=\liminf_{t\to\infty}\frac1t
\sup_{\xi_0\in\mathscr A}\int_0^t\|f'(u(t))\|_{L_d}\,dt,
$$
 where $\{u(t),\Dt u(t)\}:=S(t)\xi_0$.
The difference is that this quantity can be estimated
using the proper Strichartz norms. For instance, for
the case $d=3$, it is controlled by the Strichartz norm
$u\in L_4(T,T+1;L_{12})$ on the attractor $\mathscr A$ if
 we have no more than the quintic growth rate of $f$:
$$
|f'(u)|\le C(1+|u|^p),\ \ p\le 4.
$$
Thus, the descibed scheme works and gives the same result,
 not only for the qubic nonlinearities $f$, but also for
 the quintic ones if the control of the $L_4(L_{12})$-norm
 is known. Unfortunately, for the non-gradient quintic case,
this bound  is known nowadays for the periodic boundary conditions only,
 see \cite{SZ} (see also \cite{KSZ} for the quintic gradient case).
  On the other hand, for the sub-quintic case $p<4$, this control
  follows in a straightforward way by perturbation arguments, see
  \cite{KSZ}. Thus, our upper bounds remain valid for the sub-quintic
  case with Dirichlet/Neumann/periodic BC as well as for the critical
  quintic case with periodic BC.
\end{remark}

\section{The lower dimensional case}\label{S:sec4}
In this section, we consider the changes which should be made in
the general scheme in order to cover the cases $d=1$ and $d=2$. For this purpose
we set
\begin{equation}\label{B1B2}
\aligned
&B_1:=\sup_{\xi_0\in\mathscr A}  \|f'(u)\|_{L_\infty(\Omega)},\quad d=1,\\
&B_2:=\limsup_{t\to\infty}\sup_{\xi_0\in\mathscr A}
\frac1t\int_0^t\|f'(u(\cdot,\tau)\|_{L_\infty(\Omega)}\,d\tau,\quad d=2.
\endaligned
\end{equation}

We point out that we clearly have that $B_1<\infty$,
while the similar fact for $B_2$ will follow from
Proposition~\ref{P:Strichartz}. For the moment we may assume
that for some  fixed $t$ we have $\|f'(\cdot,t)\|_{L_\infty}<\infty$
and therefore both for $d=1$ and $d=2$ we can write for the key term in~\eqref{anyd}
\begin{equation}\label{anyd12}
\aligned
\int_\Omega  \langle f'(u(t,x))\varphi_j(x),\psi_j(x)\rangle\,dx=
\int_\Omega  \langle \varphi_j(x),f'(u(t,x))^T\psi_j(x)\rangle\,dx\\=
\left(\varphi_j,f'(u(t,\cdot))^T\psi_j\right)_{L_2}=
\left(A^{1/4}\varphi_j,A^{-1/4}f'(u(t,\cdot))^T\psi_j\right)_{L_2}\\\le
\frac\nu 2(A^{1/2}\varphi_j,\varphi_j)+
\frac1{2\nu}([f'(u(t,\cdot))A^{-1/2}f'(u(t,\cdot))^T]\psi_j,\psi_j),
\endaligned
\end{equation}
where we set $A:=-\bm\Delta^D_{\Omega}$ and $\nu>0$ is arbitrary.
Here  $-\bm\Delta^D_{\Omega}$ is the Dirichlet Laplacian
acting independently on $N$-vectors and $\bm\lambda_j$ are its eigenvalues
(see Section~\ref{S:sec7}).
Now Proposition~\ref{P:d12} gives for the first term
$$
\sum_{j=1}^n(A^{1/2}\varphi_j,\varphi_j)\le\sum_{j=1}^n\bm\lambda_j^{-1/2}.
$$
We consider the second term. The nonzero eigenvalues of the compact
self-adjoint operator $f'A^{-1/2}{f'}^T$ are the same as the eigenvalues of the
operator $A^{-1/4}{f'}^Tf'A^{-1/4}$ with quadratic form
$$
\left(A^{-1/4}{f'}^Tf'A^{-1/4}\psi,\psi\right)=
\left(f'A^{-1/4}\psi,f'A^{-1/4}\psi\right)\le\|f'\|_{L_\infty}^2(A^{-1/2}\psi,\psi).
$$
Therefore by the variational principle and Lemma~\ref{L:sub} we obtain
$$
\sum_{j=1}^n([f'(u(t,\cdot))A^{-1/2}f'(u(t,\cdot))^T]\psi_j,\psi_j)\le
\|f'(u(t)\|^2_{L_\infty}
\sum_{j=1}^n\bm\lambda_j^{-1/2}.
$$

Combining the above and optimizing in $\nu$ we obtain for the $n$-trace
the estimate (holding both for $d=1$ and $d=2$)
\begin{equation}\label{anydddd}
\sum_{j=1}^n(\mathcal L_\eb(u(t)) \Theta_j, \Theta_j)_E\le
-\frac\varepsilon 2n+\|f'(u(t)\|_{L_\infty}
\sum_{j=1}^n\bm\lambda_j^{-1/2}.
\end{equation}

This completes the preparatory work for the case $d=1$.

We now turn to the case $d=2$.
In this case we do not have the embedding
$H^1\subset C$, so only the energy norm is not enough
 for the control of the quantity $B_2$. Moreover,
 we need some  growth restrictions
on the gradient part $\nabla F_0(u)$ of the nonlinearity $f$.
For simplicity, we assume that
\begin{equation}\label{cond2d}
|\nabla^3F_0(u)|\le C(1+|u|^p),\ \ u\in\R^N
\end{equation}
where $p$ may be arbitrarily large although some exponentially
growing functions $f$ may also be included,
see, for example, \cite{Strichartz}.

As in the 1D case, let conditions
\eqref{condfb}--\eqref{22}  hold. The key
tool which allows to overcome the problem here
will be the following Strichartz type estimate for the linear wave equation
proved in~\cite{Strichartz}.

\begin{proposition}\label{P:Strichartz} Let $u$ be an energy solution of the following
linear problem:
$$
\Dt^2 u-\Delta u=g(t),\ \ u\big|_{\partial\Omega}=0,\ \xi_u\big|_{t=0}=\xi_0
$$
and let $g\in L_1(0,T;L_2)$, $\xi_0\in E$ and let $\Omega\subset \mathbb R^2$ be
 a bounded smooth domain. Then the following Strichartz estimate holds:
$$
\|u\|_{L_8(0,T;C^{1/8}(\bar\Omega))}\le C_T(\|\xi_0\|_{E}+\|g\|_{L_1(0,T;L_2(\Omega))})
$$
where the constant $C_T$ depends on $\Omega$ and $T$.
\end{proposition}
Recall that we have the Sobolev embedding $H^1\subset L_q$ for any
finite $q$, so we have the control of the $L_1(L_2)$-norm of $f(u)$
 via the energy norm:
$$
\|f(u)\|_{L_1(T,T+1;L_2)}\le C(1+\|\xi_u\|_{L_\infty(T,T+1;E)})^{p+2}.
$$
Combining this with the dissipative energy estimate
$\partial_tu\in  L_\infty(0,\infty;L_2)$, we get that
$$
\sup_{\xi_0\in\mathscr A}\|u\|_{L_8(T,T+1;L_\infty(\Omega))}\le
\sup_{\xi_0\in\mathscr A}\|u\|_{L_8(T,T+1;C^{1/8}(\bar\Omega))}\le C
$$
where $C$ is independent of $\gamma$, $T$ and  the the choice of the
initial data $\xi_0\in\mathscr A$.

Finally, using time averaging and H\"older's inequality  we see that the quantity $B_2$ in \eqref{B1B2}
is bounded by the same constant $C$.

Thus, we get the following analogue of Theorem \ref{T:main-d}
for $d=1,2$.
\begin{theorem}\label{frac-low}
Let conditions \eqref{condfb}--\eqref{22} hold, and, in addition, let
condition \eqref{cond2d}
hold for $d=2$.  Then, the fractal
dimension of the attractor is finite and possesses the
following upper bounds.

If $d=1$ then
$$
\dim_F\mathscr A\le n^*,
$$
where $n^*$ is the unique root of the equation
$$
n=N\frac8\pi\frac1\gamma\ell B_1\ln(en),
$$
which satisfies
\begin{equation}\label{d1est}
\dim_F\mathscr A\le n^*\le
N\frac{16}\pi\frac1\gamma\ell B_1\ln\left(N\frac{8}\pi\frac1\gamma\ell B_1\right),
\end{equation}

For $n=2$
\begin{equation}\label{d2est}
\dim_F\mathscr A\le N\frac{2^7}\pi\frac1{\gamma^2}|\Omega|B_2^2.
\end{equation}
\end{theorem}
\begin{proof}
Setting $\varepsilon:=\gamma/4$ in \eqref{anydddd} and using~\eqref{d1}
for $d=1$ we obtain the following bound
for the numbers $q(n)$:
$$
q(n)\le-\frac{n\gamma}8+ B_1\sum_{j=1}^n\bm\lambda_j^{-1/2}\le
-\frac{n\gamma}8+ B_1\frac{N\ell}\pi\ln en,
$$
which gives \eqref{d1est}, since the root of the equation $n=A\ln en$
for a large $A$ satisfies $n\le2A\ln A$.

Accordingly, for $d=2$ we have
$$
q(n)\le-\frac{n\gamma}8+ B_2\sum_{j=1}^n\bm\lambda_j^{-1/2}\le
-\frac{n\gamma}8+B_2\left(\frac{N|\Omega|}{2\pi}\right)^{1/2}2\sqrt n,
$$
which gives \eqref{d2est}.

\end{proof}

\begin{remark}\label{R:d1}
For $d=1$ there exists an elementary proof of
non-optimal (of order $\gamma^{-2}$) estimate for the dimension
\begin{equation}
\dim_F\mathscr A\le N\frac{16}{\gamma^2}\,\ell B_1^2
\label{d1simple},
\end{equation}

 In fact, arguing as in \eqref{anyd}
 and using
estimate \eqref{Nd1}, we obtain
$$
q(n)\le-\frac\gamma 8n+B_1(N\ell/4)^{1/2}\sqrt n,
$$
which gives \eqref{d1simple}.

\end{remark}

 \setcounter{equation}{0}
\section{Lower bound for the dimension of the attractor}\label{S:sec5}
 In this section we obtain  lower bounds for the dimension of the attractor.
We prove these estimates  for the system~\eqref{hyp}
where $\Omega\subset \mathbb R^d$, $d\ge1$ and
$u =(u_1,\dots,u_N)$, $N\ge2$. As before, the nonlinear vector function
$f(u)$ is of the form \eqref{f(u)}  satisfying \eqref{condfb}--\eqref{22}.
Next,  we just set $g=0$.

\begin{theorem}\label{T:lower}
Let in the system~\eqref{hyp} $ g=0$, $f(0)=0$,
 and let the $N\times N$ matrix $f'(0)$
 have a complex eigenvalue $a+ib$.
 Then the following  lower bound for the dimension holds
\begin{equation*}\label{lbd}
 \dim_F\mathcal A\ge\frac C{\gamma^d},
\end{equation*}
 where $C$ is independent of $\gamma$.
\end{theorem}

\begin{proof} The global attractor is strictly invariant
and therefore is a section at
$t=0$ (or at any $t=t_0$) of the set of complete
trajectories bounded for $t\in\mathbb R$. The solutions starting from
the unstable manifold of a stationary solution are obviously
bounded for $t\in\mathbb R$. Hence, the dimension of the
unstable manifold is always a lower bound for the dimension of the global
attractor.

Let us consider the trivial stationary solution  $u=0$
and the corresponding linearized system:
  $$
  \partial_t^2v+\gamma \partial_tv-\Delta v+f'(0)v=0.
  $$
Corresponding to the complex eigenvalue  $a+ib$ of the matrix $f'(0)$
 is the (constant) vector $V$.
We shall seek the solutions of the linearized system in the form
$$
v(t,x)=e^{\mu t}\varphi_n(x)V,
$$
where $\{\varphi_n\}_{n=1}^\infty$
are the eigenfunctions of the Dirichlet problem
$$
-\Delta\varphi_n=\lambda_n\varphi_n,
\qquad\varphi_n\vert_{\partial\Omega}=0.
$$
We obtain for each  $n$ the equation
$$
\mu^2+\gamma\mu+\omega^2+a+ib=0,
$$
where $\omega^2:=\lambda_n$ (and the similar equation with $a-ib$).
Then
\begin{multline*}
\mu_{\pm}=-\frac\gamma 2\pm
  \sqrt{\frac{\gamma^2}4-(a+ib+\omega^2)}=
  -\frac\gamma 2\pm\sqrt{-\omega^2
  \left(1+\frac{4a-\gamma^2}{4\omega^2}+i\frac b{\omega^2}\right)}=\\
  =-\frac\gamma 2\pm i\omega\sqrt{1+\frac{4a-\gamma^2}{4\omega^2}+
  i\frac b{\omega^2}}=-\frac\gamma 2\mp\frac b{2\omega}\pm i
  \left(\omega+\frac{4a-\gamma^2}{8\omega}\right)+
  O\left(\frac1{\omega^3}\right).
\end{multline*}

Choosing the sign of $b$ accordingly
(it is at our disposal, since the matrix is real-valued),
we see that for $\mu_+=\mu_+(n)$
it holds:
$$
\Re \mu_+=\frac b{2\omega}-\frac\gamma
 2+O\left(\frac1{\omega^3}\right),\quad b>0.
 $$
This gives that if  $\omega$ is sufficiently large, but
\begin{equation}\label{unstab}
\omega<\frac b\gamma\,,
\end{equation}
then  $\Re \mu>0$, and the corresponding solution
is growing exponentially. Finally, it follows from the
Weyl asymptotic formula that
$$
\omega_n^2=\lambda_n\sim\left(\frac{(2\pi)^d}{\omega_d|\Omega|}\right)^{2/d}n^{2/d},
\quad n\to \infty,
$$
where $\omega_d$ denotes the volume of the unit ball in $\mathbb R^d$.
This, in turn, implies that for
$\gamma\to0^+$
there exists at least $O(\gamma^{-d})$ numbers $n$,
for which  \eqref{unstab} holds. The proof is complete.
\end{proof}

\begin{remark}\label{R:Lower-bound}
{\rm
This approach for obtaining order sharp lower bounds
does not (and should not) work in the gradient case.
In fact, if $f_\gamma(u)=0$ and  $f=\nabla_uF_0$, then the Hessian matrix
$f'(u)=\nabla^2F_0(u)=\{\partial^2F_0/\partial u_i\partial u_j\}$
is symmetric and hence cannot have a complex eigenvalue.

Moreover, it is clear that we  can always find a $\gamma$-small
vector function  $f_\gamma(u)$ with matrix of derivatives
$f'_\gamma(u)$ of order $O(1)$, which produces a complex eigenvalue
for the sum
$$
\nabla^2F_0(u)+f'_\gamma(u)
$$
at $u=0$.
 For instance, we may take $N=2$, $\Phi(u)\equiv0$ and
$$
f_\gamma(u)=\gamma\(\sin(u_2/\gamma),-\sin(u_1/\gamma)\).
$$
}
\end{remark}

Thus, we have the following result.

\begin{corollary}\label{Cor:low}
Let  the nonlinearity $f$ be chosen as in Remark \ref{R:Lower-bound}.
 Then the corresponding attractor
 possesses the following lower bound for the dimension:
$$
\dim_F\mathscr A\ge C_d\gamma^{-d}.
$$
where the constant  $C_d$ is uniform with respect to $\gamma\to0^+$.
\end{corollary}

Combining Theorem~\ref{T:main-d}, Theorem~\ref{frac-low},
 and
the results of this section we can come to the following conclusion.

\begin{corollary}\label{Cor:upandlow}
The following upper and lower bounds
for the dimension of the attractor holds
$$
\aligned
d=1\quad \gamma^{-1}\preceq&\dim_F\mathscr A\preceq\gamma^{-1}\ln(\gamma^{-1}),\\
d\ge2\quad \gamma^{-d}\preceq &\dim_F\mathscr A\preceq \gamma^{-d}.
\endaligned
$$
As usual, the lower bounds hold for a specially chosen family of
functions $f(u)$.
\end{corollary}

\begin{remark}\label{R:log} In the case $d=1$ the logarithm
in the upper bound \eqref{d1est}
for the fractal dimension of the attractor cannot
be removed at least using the volume contraction method, since
 in the example  in this section
 the Lyapunov dimension at the equilibrium points behaves
 like $\gamma^{-1}\ln(1+\gamma^{-1})$ as $\gamma\to0^+$.

In fact, let there be, say, $N$ unstable eigenvalues. Then
for any $n\le N$, the summation from $1$ to $n$ of the complex
conjugate pairs
of the unstable eigenvalues gives that the Lyapunov dimension
of the equilibrium $u=0$ is the
root of the equation
$$
-\gamma n+b\frac l\pi\sum_{j=1}^n \frac1j=0,
$$
which is of order $\gamma^{-1}\ln(1+\gamma^{-1})$ as $\gamma\to0$.

\end{remark}

\section{The case of a gradient nonlinearity}\label{S:sec6}
In this section, we will discuss the case where $f_b(u)\equiv0$, so the considered equation possesses a global Lyapunov function
\begin{equation}\label{4.Lyap}
\frac d{dt}\mathscr L(\xi_u(t))=-\gamma\|\Dt u(t)\|^2_{L_2},
\end{equation}
where
$$
 \mathscr L(\xi_u)=\frac12\|\partial_tu\|^2_{L_2}+
 \frac12\|\nabla u\|^2_{L_2}+(F_0(u),1)+ (g, u),
 $$
 see e.g., \cite{BV}.
 \par
This case is principally different to the ones considered
above since, due to this Lyapunov function,
the global attractor $\mathscr A$ is a union of the
unstable sets of equilibria $\xi_k=(u_k,0)\in\mathcal R$.
Moreover, generically the set $\mathscr A$ is finite
and all equilibria are hyperbolic. In this case, the
attractor is a finite union of unstable manifolds of this equilibria
\cite{BV}:
$$
\mathscr A=\bigcup_{k=1}^M\mathcal M_+(u_k).
$$
By this reason, the Hausdorff dimension of the attractor is given by
\begin{equation}\label{4.H}
\dim_H\mathscr A=\max_{k}\dim\mathcal M_+(u_k).
\end{equation}
Using also that the operator $-\Delta +f'(u_k)$ is self-adjoint for
any $k$ (since $f'(u_k)=\nabla^2F_0(u_k)$ is a symmetric matrix) and,
therefore, all its eigenvalues are real, it is easy to see that
the index of instability of any $\xi_k\in\mathcal R$ remains bounded
as $\gamma\to0$. This, together with \eqref{4.H} ensures that the
 Hausdorff dimension of the attractor $\mathscr A$ also remains bounded as $\gamma\to0$.
\par
However, the fractal dimension $\dim_F\mathscr A$ may a priori
grow as $\gamma\to0$ due to the complicated intersections of
stable and unstable manifolds, see \cite{ZPochin} for related examples.
Thus, it is still an interesting problem to estimate the fractal
dimension of the attractor in the gradient case. We start with the Lyapunov dimension of the attractor.
The main technical tool for our study  is the
following abstract theorem.
\begin{theorem}\label{Th-dim-grad} Let the assumptions of
Definition \ref{def-Lyap} hold and let the semigroup $S(t)$
possess a continuous global Lyapunov function $\mathscr L:\mathscr A\to\R$.
By definition, this function is non-increasing along the trajectories
and the equality $\mathscr L(S(t)\xi)=\mathscr L(\xi)$ for some $\xi\in\mathscr A$
and $t>0$ implies that $\xi\in\Cal R$ is an equilibrium.
\par
Then
\begin{equation}\label{dimeq}
\dim_L(S,\mathscr A)=\sup_{\xi\in\mathscr A\cap\mathcal R}\dim_L(S,\{\xi\}).
\end{equation}
\end{theorem}
\begin{proof} Although this result is some kind of folks knowledge,
we failed to find a sharp reference and, by this reason, prefer to
 give its proof here.
\par
Note first of all that $\mathscr A\cap\Cal R\subset\mathscr A$, so only
the sign "$\le$" in \eqref{dimeq} requires a proof. Let $\xi_u(t)\in\mathscr A$,
$t\in\R_+$, be a trajectory on $\mathscr A$. { Then, for every $\eb>0$ and every $M\gg1$,
 there exist $L=L(\eb,M)$ and  sequences $T_k^\pm$ and $L_k $,
 $k\in\mathbb N$, $\xi_k\in\mathscr A\cap \Cal R$, such that
 \begin{equation}\label{seq}
 \begin{cases}
 \|\xi_u(t)-\xi_k\|_E\le\eb,\ \ t\in [T_k^-,T_k^+], \ \   T_{k+1}^-=T_k^++L_k,\\ M\le T_k^+-T_k^-\le 2M,\ \ \sum_{k}L_k\le L.
\end{cases}
 \end{equation}
 Moreover, without loss of generality, we may assume that either $L_k=0$ or $L_k\ge1$, so the total number of non-zero $L_k$ is finite.}
The existence of such sequences follows in a  straightforward
way from the properties of a Lyapunov function, see \cite{BV,VZCh}.
Here $T_k^\pm$ and $L_k$ depend on the choice of $u(t)$, but $L$ and
$M$ are independent of $u$. In addition, since the sequence $L_k$ contains
only finitely many non-zero terms, so we have the stabilization of
 $u(t)$ to $\Cal R$ as $t\to\infty$.
\par
Denote $\omega_n:=\sup_{\xi\in\mathscr A\cap\Cal R}\omega_n(S,\{\xi\})$
and for every $\delta\ll 1$ let $M=M(\delta)\gg1$ be such
that $\omega_n(\{\xi\},T)\le(\omega_n+\delta)^T$ for all
 $T\ge M$ and all $\xi\in\mathscr A\cap\Cal R$.
{
 Indeed, assume that such $M$ does not exist, then there exists
 sequences $\xi_k\in\mathscr A\cap\mathcal R$, $M_k\to\infty$ and
 $\delta_0>0$ such that
 $$
 \omega_n(\{\xi_k\},M_k)\ge(\omega_n+\delta_0)^{M_k}.
 $$
 Without loss of generality, we assume that
 $\xi_k\to\xi_0\in\mathscr A\cap\mathcal R$ and,
   since the function $T\to\sqrt[T]{\omega_n(\{\xi\},T)}$ is non-increasing, we have
   $$
   \omega_n(\{\xi_k\},M)\ge(\omega_n+\delta_0)^{M}
 $$
 for every  $M$ and every sufficiently large $k$. The continuity of
 the norm $\xi\to\omega_n(\{\xi\},M)$ now gives
 $$
 \omega_n(\{\xi_0\},M)\ge(\omega_n+\delta_0)^{M},\ \forall M,
 $$
 which contradicts the definition of $\omega_n$.
 \par
  Using the continuity
of norms $\omega_n$, together with inequalities \eqref{seq}, we see that
\begin{multline*}
\omega_n(S'(S(T_k^-)\xi,T_k^+-T_k^-))\le
\omega_n(S'(\xi_k,T_k^+-T_k^-))+A_M(\eb)\le\\
\le \(\omega_n+\delta+A_M(\eb)^{1/(2M)}\)^{T_k^+-T_k^-},
\end{multline*}
where $\lim_{\eb\to0}A_M(\eb)=0$ for every fixed $M$. Taking a
 big number $T\gg1$ and using sub-multiplicativity, we end up with
\begin{multline*}\label{KL}
\omega_n(S'(\xi,T))\le e^{Kn(L+2M+1)}\(\omega_n+\delta+A_M(\eb)^{1/(2M)}\)^{\sum_{k=1}^{k(T)}(T_{k}^+-T_{k}^-)}\le\\\le
\delta^{-L-2M-1}e^{2nK(L+2M+1)}\(\omega_n+\delta+A_M(\eb)^{1/(2M)}\)^{T},
\end{multline*}
where $k(T):=\max\{k\in\mathbb N,\, T_{k}^+<T-1\}$ and $K>1$ is such that
$$
\max_{\xi\in\mathcal A}\|S'(\xi,t)\|_{\mathcal L(E,E)}\le e^{Kt}
$$
for all $t\ge1$. Such number exists due to sub-multicativity (since the root of power $t$ from the left-hand side of the last formula is a non-increasing and bounded for $t\ge1$ function of time).
We have also implicitly used here that
$$
\sum_{k=1}^{k(T)}(T_{k}^+-T_{k}^-)\ge T-L-2M-1\ \ \text{ and }\ \ \omega_n+\delta+A_M(\eb)^{1/(2M)}\ge\delta.
$$}
Taking the supremum with respect to $\xi\in\mathscr A$
together with the root of power $T$, we end up with the inequality
$$
\omega_n(S,\mathscr A)\le \omega_n+\delta+A_M(\eb)^{1/(2M)}.
$$
Since $M=M(\delta)$ is independent of $\eb$, we may pass to the
limit $\eb\to0$ and get
$$
\omega_n(S,\mathscr A)\le \omega_n+\delta
$$
and finally passing to the limit $\delta\to0$, we end up
with $\omega_n(S,\mathscr A)\le\omega_n$ which finishes the proof of the theorem.
\end{proof}
We now return to our damped wave equation with a gradient non-linearity.
 In this case  equation \eqref{4.Lyap}  holds,
 (we emphasize that the condition of
 hyperbolicity of all equilibria is not posed here),
  so all  the assumptions of the theorem are satisfied and we have the estimate
$$
\dim_F\mathscr A\le \dim_L(S,\mathscr A)= \sup_{\xi\in\mathcal R }\dim_L(S,\{\xi\}).
$$
Moreover, since the Lyapunov dimension of an equilibrium can be easily expressed
through the
 corresponding eigenvalues, see e.g. \cite{T, KR}, an
 elementary calculation gives
\begin{equation}\label{dimLyap1}
\dim_L(S,\{\xi\})\le C\gamma^{-1},
\end{equation}
where $C$ may depend on $d$, but is independent of $\gamma\to0$.

Indeed, let $\xi=(\bar u,0)$ be a stationary solution of
our damped wave equation. Then
$$
-\Delta \bar u+f(\bar u)=g, \quad \bar u\vert_{\partial\Omega}=0,
$$
and we consider the corresponding linearized
equation
$$
\partial_t^2u+\gamma \partial_tu=\Delta u+a(x)u,\quad a(x)=f'(\bar u(x)),
$$
or
$$
\partial_ty=Ay, \quad A=\left(
                                 \begin{array}{cc}
                                   0 & I\\
                                -L_a& -\gamma \\
                                 \end{array}
                               \right),
$$
where $y=(u,\partial_tu)$, $L_au=-\Delta u+a(x)u$, $u\vert_{\partial\Omega}=0$.

Let $\{\nu_j\}_{j=1}^\infty$, $\nu_j\to\infty$, be the non-decreasing sequence
 of eigenvalues of the operator $L$. Then the point spectrum of the operator
 $A$ can be expressed in terms of the $\nu_j$'s as follows
 (see \cite[Theorem~IV.4.5]{BV}):
it consists of the two sequences
\begin{equation}\label{mus}
\mu^1_{j}=(-\gamma-\sqrt{\gamma^2-4\nu_j})/2,\quad
\mu^2_{j}=(-\gamma+\sqrt{\gamma^2-4\nu_j})/2.
\end{equation}

We now estimate
the numbers $\omega_n(S,{(\bar u,0)})$, see \eqref{q(n)}.
Taking into account~\eqref{mus}  we obtain
\begin{multline*}
\omega_n(S,\{(\bar u,0)\})=-\frac{n\gamma}2+\\+\frac12\sum_{j=1}^n\sqrt{\gamma^2-4\nu_j}
\le-\frac{n\gamma}2+\sum_{j=1}^n\sqrt{-\nu_j}
=:-\frac{n\gamma}2+C_{\bar u},
\end{multline*}
which proves \eqref{dimLyap1} for $\xi=(\bar u,0)$.

 The uniformity of this estimate with
respect to $\xi\in\mathcal R$ follows in a standard way from
the min-max principle by bounding the operator $L_a$ from below
by the operator $L_C:=-\frac12\Delta-C$ and choosing the proper values of $C$
in order to make this bound uniform with respect to $\xi\in\mathcal R$.

Indeed, let for simplicity $d\ge3$. Then, due to the growth restriction on $f'$,
unform bounds of $\mathcal R$ in $H^1_0(\Omega)$ and the Sobolev inequality
$$
\|v\|_{L_{2d/(d-2)}}\le\mathrm S_d\|\nabla v\|_{L_2},
$$
we have the bound  $\|a\|_{L_d}\le c_0$ for the
norms  $\|a\|_{L_d}$ that is  uniform with respect to $\gamma\to0$ and $\mathcal R$
(in the case $d=1$ we will have bounds for $L_\infty$ and for $d=2$,
 the estimate is true in $L_p$ for all $p<\infty$).  Therefore,
 by the  H\"older  and Sobolev inequalities
\begin{multline*}
(L_a v,v)=\|\nabla v\|^2_{L_2}+(a v,v)\ge \|\nabla v\|^2_{L_2}-
\mathrm S_d\|a\|_{L_d}\|v\|_{L_2}\|\nabla v\|_{L_2}
\ge\\\ge \frac12\|\nabla v\|^2_{L_2}-\frac12\mathrm S_d^2\|a\|_{L_d}^2\|v\|^2_{L_2}\ge
\frac12\|\nabla v\|^2_{L_2}-\frac12\mathrm S_d^2c_0^2\|v\|^2_{L_2},
\end{multline*}
which gives us a uniform lower bound for the operators $L_a$
and hence
 the desired  upper bound for the Lyapunov dimension.
\medskip

Combining the obtained estimates and Remark~\ref{R:log}
 we obtain the following result.
\begin{corollary} Let the assumptions of Theorem~\ref{Th-dim-grad}
 be satisfied and let
 $f_\gamma\equiv0$.   Then
$$
\dim_F\mathscr A\le\dim_L(S,\mathscr A)\le
   C\gamma^{-1}, \ d\ge1,
$$
where the corresponding constant $C$ is independent of $\gamma\to0$.
\par
Moreover, let there exist an equilibrium $\xi\in\mathcal R$ with
strictly negative eigenvalue of the corresponding operator $L_a$
(this means that the attractor is not trivial). Then the Lyapunov dimension
possesses two-sided optimal bounds of the same order:
$$
C_1\gamma^{-1}\le \dim_L(S,\mathscr A)\le C_2\gamma^{-1},
\quad d\ge1,
$$
where $C_1>0$ and $C_2>0$ are independent of $\gamma\to0$.
\par
Finally,  assume that all  the equilibria $\xi\in\Cal R$
 are hyperbolic (then their number is automatically finite).
Then the Hausdorff dimension of the attractor does not grow as
$\gamma\to0$  and we have two-sided estimates
$$
C_1\sim\dim_H\mathscr A\le\dim_F\mathscr A\le\dim_L(S,\mathscr A)\sim C_2\gamma^{-1}.
$$
\end{corollary}

Indeed, the Hausdorff dimension of an unstable manifold
equals  the instability index of the corresponding
equilibrium and this index, in turn, is bounded  from above by
the number of the unstable eigenvalues of the corresponding operator $L_C$
which is independent of $\gamma\to0$.

\begin{remark} We see that the obtained estimates are sharp for
both Hausdorff and Lyapunov dimensions, but there is still an
essential gap between upper and lower bounds for the fractal
dimension. This raises  an interesting question about the
fractal dimension for the attractor of a gradient system
(or more general, for a system with the global Lyapunov function)
when all of the equilibria are hyperbolic. Namely, could the
fractal dimension be bigger than the Hausdorff dimension in that case?
The affirmative answer on this question
is given in \cite{ZPochin} for the case of $C^k$-smooth gradient
systems for any finite $k$, but the case of $C^\infty$ or real
analytic gradient systems remain completely open. Another
interesting related question is about the case where stable
and unstable manifolds of equilibria intersect transversally
(the Morse--Smale case). We make a conjecture that in this case
the fractal dimension of the attractor coincides with the Hausdorff dimension.
\end{remark}
 \setcounter{equation}{0}
\section{Estimates for systems with suborthonormal gradients}\label{S:sec7}

In many applications (including ours) one deals with
vector systems or many-component systems that are orthonormal in
a Hilbert space $H$. For instance,
let the system $\{\varphi_i,\psi_j\}_{i=1}^n$ satisfy
$$
(\varphi_i,\varphi_j)+(\psi_i,\psi_j)=\delta_{i\,j},\quad i,j=1,\dots, n,
$$
or, equivalently, for all $\xi\in\mathbb R^n$
$$
\sum_{i,j=1}^n(\varphi_i,\varphi_j)\xi_i\xi_j+\sum_{i,j=1}^n(\psi_i,\psi_j)\xi_i\xi_j=
\sum_{j=1}^n\xi_j^2.
$$
Since
$$
\sum_{i,j=1}^n\xi_i\xi_j(\varphi_i,\varphi_j)=
\left\|\sum_{i=1}^n\xi_i\varphi_i\right\|^2\ge0
$$
 and similarly for $\{\psi_j\}_{j=1}^n$, both sums on the left-hand side are non-negative, and therefore
\begin{equation}\label{sub}
\sum_{i,j=1}^n(\varphi_i,\varphi_j)\xi_i\xi_j\le\sum_{j=1}^n\xi_j^2,
\end{equation}
and similarly for $\{\psi_j\}_{j=1}^n$.

A system $\{\varphi_j\}_{j=1}^n$ satisfying inequality ~\eqref{sub}
for every $\xi\in \mathbb R^n$ is called suborthonormal. This is a
useful and flexible generalization of orthonormality  (see
\cite{GMT}, where it
  probably appeared for the first time).

\begin{lemma}\label{L:sub}
Let ${K}$ be a compact self-adjoint positive operator in a
Hilbert space $H$ with spectrum
$$
\mathrm{K}e_i=\mu_ie_i,\quad \mu_1\ge\mu_2\ge\dots\to0,\quad (e_i,e_j)=\delta_{i\,j}.
$$
Then for any suborthonormal system $\{\varphi_i\}_{i=1}^n$
\begin{equation}\label{varsub}
\sum_{i=1}^n(\mathrm K\varphi_i,\varphi_i)\le\sum_{i=1}^n\mu_i.
\end{equation}
\end{lemma}
\begin{proof}
We point out that for an orthonormal system $\{\varphi_i\}_{i=1}^n$
inequality~\eqref{varsub} immediately follows from the variational principle.

Now let $P$ be the orthogonal projection onto $\mathrm{Span}\{e_1,\dots,e_n\}$.
There exists a bounded operator $O$, such that
$$
\varphi_i=Oe_i,\quad i=1,\dots,n,
$$
and we set $Oe_i=0$ for $i>n$.

Therefore
$(\varphi_i,\varphi_j)=(Be_i,e_j)$, where
 $B:=O^*O$, and  for an $a=\sum_{j=1}\xi_ie_i$
 with $\sum_{j=1}^n\xi^2_j=1$ we have by suborthonormality that
$$
(Ba,a)=\sum_{i,j=1}^n\xi_i\xi_j(Be_i,e_j)=
\sum_{i,j=1}^n\xi_i\xi_j(\varphi_i,\varphi_j)\le\sum_{j=1}^n\xi_j^2=\|a\|^2.
$$
This implies that
$$
1\ge\|B\|=\|O\|=\|O^*\|,
$$
and using  cyclicity of the trace this finally gives that
$$
\aligned
\sum_{i=1}^n(K\varphi_i,\varphi_i)=
\sum_{i=1}^n( PO^*K OPe_i,e_i)=\Tr(PO^*K OP)=\\=
\Tr(OPO^*K)=\sum_{i=1}^n( OPO^*Ke_i,e_i)=
\sum_{i=1}^n\mu_i( OPO^*e_i,e_i)\le\sum_{i=1}^n\mu_i,
\endaligned
$$
since $\|OPO^*\|\le1$.
\end{proof}

The next theorem collects the inequalities for systems with suborthonormal gradients
 that were used in Section~\ref{S:sec3}.  The scalar case is
 singled out and
 considered first. The inequalities used in Theorem~\ref{frac-low} in
  the cases $d=1$ and $d=2$ are treated separately in Proposition~\ref{P:d12}.

\begin{theorem}
\label{T:Lieb_suborth} Let $\{\varphi_i\}_{i=1}^n\in
H^1_0(\Omega)$, $\Omega\subset R^d$ make up a system of scalar
functions with suborthonormal  gradients in $L_2(\Omega)$:
\begin{equation}\label{suborth}
\sum_{i,j=1}^n\xi_i\xi_j(\nabla\varphi_i,\nabla\varphi_j)\le
\sum_{j=1}^n\xi_j^2.
\end{equation}
Then  the function
$$
\rho(x):=\sum_{j=1}^n|\varphi_j(x)|^2
$$
satisfies the following inequalities.

If $d=1$ and $\Omega=(0,\ell)$, then
\begin{equation}\label{d1sharp}
\|\rho\|_{L_\infty}\le\frac \ell 4\,,
\end{equation}
where the constant is sharp.

If $d=2$, and $|\Omega|<\infty$, then
\begin{equation}\label{d22}
\|\rho\|_{L_1}\le\frac{|\Omega|}{2\pi}\ln(en).
\end{equation}

If $d\ge3$,  $p=d/(d-2)$ and $\Omega\subseteq\mathbb R^d$  is an arbitrary domain, then
\begin{equation}\label{d3}
\|\rho\|_{L_p}\le L_{0,d}^{2/d}\frac d{d-2}n^{(d-2)/d},
\end{equation}
where $L_{0,d}$ is the constant in the Cwikel--Lieb--Rozenblum
bound for the number $N(0,-\Delta-V)$ of  negative eigenvalues of
the Schr\"odinger operator $-\Delta-V$, $V(x)\ge0$ in $\mathbb
R^d$, see~\cite{Cwikel,L,R}:
\begin{equation}\label{CLRbound}
N(0,-\Delta-V)\le L_{0,d}\int_{\mathbb R^d}V(x)^{d/2}dx.
\end{equation}

Next, if  vector functions
$\varphi_i=(\varphi_i^1,\dots,\varphi_i^N)$
make up a system
$\{\varphi_i\}_{i=1}^n\in \bm H^1_0(\Omega)$ with suborthonormal gradients in
$\bm L_2(\Omega)$ in the sense of \eqref{suborth}, then
the following inequalities hold:

\begin{align}
&1) \ d=1\quad \|\rho\|_{L_\infty}\le N\frac \ell 4, \label{Nd1}\\
&2)\ d=2\quad \|\rho\|_{L_1}\le  N \frac{|\Omega|}{2\pi}\ln(en), \label{Nd2}\\
&2) \ d\ge 3,  \ p=d/(d-2)\quad
\|\rho\|_{L_p}\le (NL_{0,d})^{2/d}\frac d{d-2}n^{(d-2)/d}, \label{Nd3}
\end{align}

\end{theorem}
\begin{proof} We first consider the scalar case.
For $d=1$  there exists a very simple proof of~\eqref{d1}, and the idea
belongs to C.\,Foias (as is acknowledged in \cite[p.440]{T}, see also~\cite{E-F}).
Using inequality~\eqref{sharp} from Lemma~\ref{L:sharp} below, for an
arbitrary $x\in(0,\ell)$ and $u=\sum_{j=1}^n\xi_j\varphi_j$
we have
$$
\biggl(\sum_{j=1}^n\xi_j\varphi_j(x)\biggr)^2\le\|u\|^2_{L_\infty}\le
\frac \ell 4\sum_{i,j=1}^n\xi_i\xi_j(\varphi_i',\varphi_j')\le\frac \ell 4\sum_{j=1}^n\xi_j^2.
$$
Setting $\xi_j:=\varphi_j(x)$, we obtain~\eqref{d1sharp}.

In the case $d=2$ we observe that
$$
(\nabla\varphi_i,\nabla\varphi_j)=(\varphi_i,-\Delta\varphi_j)=
((-\Delta^D_\Omega)^{1/2}\varphi_i,(-\Delta^D_\Omega)^{1/2}\varphi_j),
$$
where $\Delta^D_\Omega$ is the Dirichlet Laplacian in $\Omega$.
Setting $\theta_j=(-\Delta^D_\Omega)^{1/2}\varphi_j$ we see that
the system $\{\theta_j\}_{j=1}^n$ is suborthonormal in $L_2$. It
now follows from Lemma~\ref{L:sub} that
$$
\|\rho\|_{L_1}=\sum_{j=1}^n(\varphi_j,\varphi_j)=
\sum_{j=1}^n((-\Delta^D_\Omega)^{-1}\theta_j,\theta_j)\le
\sum_{j=1}^n\lambda_j^{-1},
$$
where $\lambda_j$ are the non-decreasing  eigenvalues of
$-\Delta^D_\Omega$, for which the Li--Yau lower bound~\cite{Li-Yau}
for $d=2$
\begin{equation} \label{LiYauscal}
\sum_{j=1}^n\lambda_j\ge\frac{2\pi}{|\Omega|}n^2,
\end{equation}
gives that $\lambda_j\ge\frac{2\pi}{|\Omega|}j$. This
gives~\eqref{d22}, since $\sum_{j=1}^nj^{-1}<\ln n+1$.

It remains to consider the case $d\ge3$. By the Birman--Schwinger
principle (see, \cite{lthbook} for the detailed treatment)
$$N(0,-\Delta-V)=n(V),
$$
where $n(V)$ is
 the number of the
eigenvalues $\mu_j\ge1$ of the operator
$$
V^{1/2}(-\Delta)^{-1}V^{1/2}=HH^*,
$$
where
$$
H=V^{1/2}(-\Delta)^{-1/2},\quad H^*=(-\Delta)^{-1/2}V^{1/2}.
$$
Next, it follows from the Sobolev inequality that the operator
$(-\Delta)^{-1/2}$ (the Riesz potential) is bounded from $L_2$ to
$L_{2p}$, and, by duality, from $L_{(2p)'}$ to $L_2$,
$(2p)'=2d/(d+2)$.

By H\"older's inequality this implies that both $H$ and $H^*$ are
bounded from $L_2$ to $L_2$. Furthermore they are compact and weak
type estimates for their $s$-numbers is the key result in~\cite{Cwikel} implying
inequality~\eqref{CLRbound}.

The operators $H^*H$  and $HH^*$ have the same sequence of non zero
eigenvalues $\mu_1\ge\mu_2\ge\dots\to0$, which, in
addition, depend homogeneously on $V$: $\mu_j(\alpha
V)=\alpha\mu_j(V)$, $\alpha>0$. Therefore
inequality~\eqref{CLRbound} gives that
$$
n(V)= j\Rightarrow \mu_1\ge\dots\ge\mu_j\ge1\Rightarrow n(V/\mu_j)=j
\le\mu_j^{-d/2} L_{0,d}\|V\|_{L_{d/2}}^{d/2},
$$
or
\begin{equation}\label{lamj}
\mu_j\le j^{-2/d} L_{0,d}^{2/d}\|V\|_{L_{d/2}}.
\end{equation}

Turning to the proof of~\eqref{d3} we extend $\varphi_j$ by
zero to the whole of $\mathbb R^d$ and denote this extension by
$\widetilde\varphi_j$. Since
$$
(\nabla\varphi_i,\nabla\varphi_j)=
(\nabla\widetilde\varphi_i,\nabla\widetilde\varphi_j)=
((-\Delta)^{1/2}\widetilde\varphi_i,(-\Delta)^{1/2}\widetilde\varphi_j),
$$
it follows that the system $\{\psi_j\}_{j=1}^n$, where
$\psi_j:=(-\Delta)^{1/2}\widetilde\varphi_j$, is suborthonormal in
$L_2(\mathbb R^d)$.

Setting $K=H^*H$ and $V=\widetilde\rho\,^{p-1}\in L_{d/2}$, where
$\widetilde\rho(x)=\sum_{j=1}^n|\widetilde\varphi_j(x)|^2$, we find
using Lemma~\ref{L:sub} and inequality~\eqref{lamj} that
$$
\aligned
\|\widetilde\rho\|_{L_p}^p=\int_{\mathbb R^d}V(x)\widetilde\rho(x)dx=
\sum_{j=1}^n\|H\psi_j\|^2\\=
\sum_{j=1}^n(K\psi_j,\psi_j)\le\sum_{j=1}^n\mu_j\le
 L_{0,d}^{2/d}\|V\|_{L_{d/2}}\sum_{j=1}^n j^{-2/d}\\
 \le \|V\|_{L_{d/2}}L_{0,d}^{2/d}\frac d{d-2}\,n^{(d-2)/d},
\endaligned
$$
which completes the proof, since
$\|\widetilde\rho\|_{L_p}=\|\rho\|_{L_p}$ and
$\|V\|_{L_{d/2}}=\|\widetilde\rho\|_{L_p}^{p-1}$.

We now consider the vector case. If a system of vector functions is suborthonormal,
then each scalar family of the corresponding components is also suborthonormal.
If  $d=1$, we apply \eqref{d1} for each of the $N$ systems $\{\varphi_i^k\}_{i=1}^n$,
$k=1,\dots,N$ and then add up the results to obtain \eqref{Nd1}.

In  the case $d=2$ we only need the lower bound~\eqref{LiYauvec}
for the eigenvalues $\bm \lambda_j$ of the Dirichlet Laplacian
acting independently on $N$-vectors. Writing $n=Nk+p$, $0\le p<N$,
using the Li--Yau bound~\eqref{LiYauscal} and the convexity of the
function $k\to k^2$ we obtain
$$
\aligned
\sum_{j=1}^n\bm\lambda_j=p\sum_{j=1}^{k+1}\lambda_j+
(N-p)\sum_{j=1}^{k}\lambda_j\ge\\\ge
\frac{2\pi}{|\Omega|}N\left(\frac pN(k+1)^2+\frac {N-p}Nk^2\right)\ge\\\ge
\frac{2\pi}{|\Omega|}N(k+p/N)^2=\frac{2\pi}{N|\Omega|}n^2.
\endaligned
$$
Since $\bm\lambda_j$'s are non-decreasing, this gives that
\begin{equation}\label{LiYauvec}
\bm\lambda_j\ge\frac{2\pi}{N|\Omega|}j
\end{equation} and hence~\eqref{Nd2}.

Finally, for $d\ge3$ we consider the Schr\"odinger operator
$\mathscr H=-\bm\Delta-VI$ acting in $\mathbb R^d$ on $N$-vectors
and the corresponding operators
$$
H=V^{1/2}(-\bm\Delta)^{-1/2},\quad
H^*=(-\bm\Delta)^{-1/2}V^{1/2}.
$$
Then the number of negative eigenvalues of $\mathscr H$ is clearly equal to $Nn(V)$,
and the eigenvalues $\bm\mu_j\ge1$ of $H^*H$ or $HH^*$
are just the eigenvalues $\mu_j$, each repeated $N$ times.
Writing for short  \eqref{lamj} in the form $\mu_j\le cj^{-\alpha}$
we have for $n=Nk+p$ with $p>0$
$$
\bm\mu_n=\mu_{k+1}\le c(k+1)^{-\alpha}<c((Nk+p)/N)^{-\alpha}<N^\alpha cn^{-\alpha},
$$
while for $p=0$ we  have $\bm \mu_n=\mu_k\le N^\alpha cn^{-\alpha}$.

Hence
$$
\bm\mu_j\le N^{2/d}j^{-2/d} L_{0,d}^{2/d}\|V\|_{L_{d/2}}
$$
and we complete the proof as in the scalar case to obtain \eqref{Nd3}.
\end{proof}

\begin{remark}\label{R:sharp rate}
{
Our lower bounds for the fractal dimension additionally show that the
rate of growth of the factor $n^{(d-2)/d}$ in \eqref{d3} is sharp
at least in the power scale, since otherwise there would have been a contradiction
with the lower bound for the dimension.
}
\end{remark}

The corresponding  inequalities used in the cases $d=1,2$
in Theorem~\ref{frac-low} are
collected below.
\begin{proposition}\label{P:d12}
Let $\Omega\subset \mathbb R^d$, $d=1,2$ and let
the
  vector functions
$\varphi_i=(\varphi_i^1,\dots,\varphi_i^N)$
make up a system
$\{\varphi_i\}_{i=1}^n\in \bm H^1_0(\Omega)$ with suborthonormal gradients in
$\bm L_2(\Omega)$ in the sense of \eqref{suborth}.

Then
the following inequalities hold:
\begin{equation}\label{d1}
\sum_{i=1}^n((-\bm\Delta^D_\Omega)^{1/2}\varphi_i,\varphi_i)
\le\sum_{j=1}^n\bm\lambda_j^{-1/2}\le
\left\{
      \begin{array}{ll}
        \frac{N\ell}\pi\ln en, & d=1; \\ \\
        \left(\frac{N|\Omega|}{2\pi}\right)^{1/2}2\sqrt n, & d=2.
      \end{array}
        \right.
\end{equation}
\end{proposition}
\begin{proof}
Setting $\eta_j:=(-\bm \Delta^D_\Omega)^{1/2}\varphi_j$, where
$-\bm\Delta_\Omega^D$ is the Dirichlet Laplacian
with eigenvalues $\{\bm\lambda_j\}_{j=1}^\infty=\{\lambda_i,\dots,\lambda_i\}_{i=1}^\infty$
acting independently on $N$-vectors, we see that
the system $\{\eta_j\}_{j=1}^n$ is suborthonormaal in $\bm L_2$, since
$(\eta_i,\eta_j)=(\nabla\varphi_i,\nabla\varphi_j)$.

If $d=1$ and $\Omega=(0,\ell)$, then $\lambda_n=(\pi/\ell)^2n^2$, which clearly
gives that
$$
\bm\lambda_j\ge\frac{\pi^2}{\ell^2}\left(\frac jN\right)^2,
$$
and the case $d=1$ follows.

The case $d=2$ follows from~\eqref{LiYauvec}.
\end{proof}

\begin{lemma}\label{L:sharp}
For a function  $u\in H^1_0(0,\ell)$ it holds that
\begin{equation}\label{sharp}
\|u\|^2_\infty\le\frac \ell 4\|u'\|^2,
\end{equation}
where the constant is sharp and the unique
(up to a constant multiple)
extremal function is
$$
u(x)=\left\{
       \begin{array}{ll}
         x, & 0\le x\le \ell/2, \\
         \ell-x, & l/2\le x\le \ell.
       \end{array}
     \right.
$$
\end{lemma}
\begin{proof}
It is highly likely that this inequality is known, but  for the sake
of completeness we shall prove it following~\cite{IZ}.
By scaling it suffices to prove \eqref{sharp} for  $\ell=1$.
Let $A$ be the operator  $-d^2/dx^2$ with Dirichlet boundary conditions
$u(0)=u(1)=0$. Its Green's function is
$$
G(x,\xi)=\left\{
           \begin{array}{ll}
             (1-\xi)x,\ x\le\xi, \\
             (1-x)\xi,\ \xi\le x.
           \end{array}
         \right.
$$
For an arbitrary $\xi\in(0,1)$ by the definition of the Green's function
and using that $A$ is positive definite, we obtain by the Cauchy--Schwartz
inequality
$$
\aligned
u(\xi)^2&=(G(\cdot,\xi),Au)^2=
(A^{1/2}G(\cdot,\xi),A^{1/2}u)^2\le\|(A^{1/2}G(\cdot,\xi)\|^2
\|A^{1/2}u\|^2=\\&=
(AG(\cdot,\xi),G(\cdot,\xi))(Au,u)=G(\xi,\xi)\|u'\|^2=
(1-\xi)\xi\|u'\|^2\le\frac14\|u'\|^2,
\endaligned
$$
where the first inequality turns into the equality if
$u(x)=cG(x,\xi)$.
\end{proof}

\begin{remark}\label{R:CLR_constant}
{\rm The constant $L_{0,d}$ is traditionally compared with its
semiclassical lower  bound
$$
L_{0,d}\ge L_{0,d}^{\mathrm{cl}}:=\frac{\omega_d}{(2\pi)^d}\,.
$$
The best to date bound for $L_{0,3}$ is Lieb's bound~\cite{L}
$$L_{0,3}\le 6.8693\cdot L_{0,3}^{\mathrm{cl}}=0.116\dots.
$$
For the recent progress in higher dimensions see \cite{Hund}.
A short proof of inequality \eqref{CLRbound} with good estimates of the constants is
 given in~\cite{FrankJST}.
}
\end{remark}

\begin{remark}\label{R:Lieb}
{\rm
For $d=1$ and $d\ge3$ we follow \cite[Theorem 1]{LiebJFA}
adapting the proof to the suborthonormal case by means of Lemma~\ref{L:sub} and specifying the constants
for $d\ge3$.
}
\end{remark}

\subsection*{Acknowledgement}
This work was done with the financial support from the Russian
Science Foundation (grant  no. 23-71-30008).

\end{document}